\documentclass[12pt,reqno]{amsart}
\usepackage{amsmath, amsthm, amssymb, stmaryrd}

\advance \topmargin by -\headheight
\advance \topmargin by -\headsep
     
\evensidemargin \oddsidemargin
\newtheorem{theorem}{Theorem}[section]

\newtheorem{corollary}[theorem]{Corollary}

\theoremstyle{definition}

\theoremstyle{remark}

\numberwithin{equation}{section}

\def\bfa{{\mathbf a}}

 \def\bfe{{\mathbf e}} 
\def\bfg{{\mathbf g}}
\def\bfh{{\mathbf h}}

\def\bfx{{\mathbf x}}
\def\bfy{{\mathbf y}}
\def\bfz{{\mathbf z}}

\def\calR{{\mathcal R}}

\def\dbC{{\mathbb C}}

\def\dbN{{\mathbb N}}  

\def\dbZ{{\mathbb Z}}\def\dbQ{{\mathbb Q}}

\def\alp{{\alpha}}  
\def\bet{{\beta}}

\def\tet{{\theta}}  
\def\Tet{{\Theta}}

\def\sig{{\sigma}}  

\def\Ups{{\Upsilon}} \def\bfUps{{\boldsymbol \Ups}}
 \def\bfvarphi{{\boldsymbol \varphi}}

\def\ome{{\omega}}

\def\eps{\varepsilon}

\def\le{\leqslant} \def\ge{\geqslant}

\begin{document}
\title[Paucity for symmetric Diophantine equations]{The paucity problem for certain 
symmetric Diophantine equations}
\author[T. D. Wooley]{Trevor D. Wooley}
\address{Department of Mathematics, Purdue University, 150 N. University Street, 
West Lafayette, IN 47907-2067, USA}
\email{twooley@purdue.edu}
\subjclass[2010]{11D45, 11D72}
\keywords{Paucity, symmetric polynomials, Diophantine equations.}
\date{}
\dedicatory{}
\begin{abstract}Let $\varphi_1,\ldots ,\varphi_r\in \dbZ[z_1,\ldots z_k]$ be integral linear 
combinations of elementary symmetric polynomials with $\text{deg}(\varphi_j)=k_j$ 
$(1\le j\le r)$, where $1\le k_1<k_2<\ldots <k_r=k$. Subject to the condition 
$k_1+\ldots +k_r\ge \tfrac{1}{2}k(k-1)+2$, we show that there is a paucity of 
non-diagonal solutions to the Diophantine system $\varphi_j(\bfx)=\varphi_j(\bfy)$ 
$(1\le j\le r)$.
\end{abstract}
\maketitle

\section{Introduction} Our focus in this memoir lies on systems of simultaneous 
Diophantine equations defined by symmetric polynomials. Let 
$\varphi_1,\ldots ,\varphi_r\in \dbZ[z_1,\ldots ,z_k]$ be symmetric polynomials with 
$\text{deg}(\varphi_j)=k_j$ $(1\le j\le r)$, so that for any permutation $\pi$ of 
$\{1,2,\ldots ,k\}$, one has
\[
\varphi_j(x_{\pi 1},\ldots ,x_{\pi k})=\varphi_j(x_1,\ldots ,x_k)\quad (1\le j\le r).
\] 
Given any such permutation $\pi$, the system of Diophantine equations
\begin{equation}\label{1.1}
\varphi_j(x_1,\ldots ,x_k)=\varphi_j(y_1,\ldots ,y_k)\quad (1\le j\le r)
\end{equation}
plainly has the trivial solutions obtained by putting $y_i=x_{\pi i}$ $(1\le i\le k)$. Denoting 
by $T_k(X)$ the number of these trivial solutions with $1\le x_i,y_i\le X$ $(1\le i\le k)$, one 
is led to the question of whether there is a paucity of non-diagonal solutions. Thus, fixing 
$k$ and $\bfvarphi$, and writing $N_k(X;\bfvarphi)$ for the number of solutions of the 
system \eqref{1.1} in this range, one may ask whether as $X\rightarrow \infty$, one has
\[
N_k(X;\bfvarphi)=T_k(X)+o(T_k(X)),
\]
or equivalently, whether $N_k(X;\bfvarphi)=k!X^k+o(X^k)$. This question has been 
examined extensively in the diagonal case where the polynomials under consideration take 
the shape $\varphi_j(\bfx)=x_1^{k_j}+\ldots +x_k^{k_j}$, as can be surmised from the 
references to this memoir. Relatively little consideration has been afforded to more general 
symmetric polynomials. It transpires that by adapting a strategy applied previously in a 
special case of Vinogradov's mean value theorem (see \cite{VW1997}), we are able to 
settle this paucity problem for numerous systems of the type \eqref{1.1} in a particularly 
strong form.\par

Further notation is required to describe our conclusions. Define the elementary symmetric 
polynomials $\sig_j(\bfz)\in \dbZ[z_1,\ldots ,z_k]$ for $j\ge 0$ by means of the generating 
function identity
\begin{equation}\label{1.2}
\sum_{j=0}^k\sig_j(\bfz)t^{k-j}=\prod_{i=1}^k(t+z_i).
\end{equation}
We restrict attention primarily to symmetric polynomials of the shape
\begin{equation}\label{1.3}
\varphi_j(\bfz)=\sum_{l=1}^ka_{jl}\sigma_l(\bfz)\quad (1\le j\le r),
\end{equation}
with fixed coefficients $a_{jl}\in \dbZ$ for $1\le j\le r$ and $1\le l\le k$. By taking 
appropriate integral linear combinations of the polynomials $\varphi_1,\ldots ,\varphi_r$, it 
is evident that in our investigations concerning $N_k(X;\bfvarphi)$, we may suppose that 
$\text{deg}(\varphi_j)=k_j$ $(1\le j\le r)$, where the exponents $k_j$ satisfy the condition
\begin{equation}\label{1.4}
1\le k_1<k_2<\ldots <k_r=k.
\end{equation}
This can be seen by applying elementary row operations on the reversed $r\times k$ matrix 
of coefficients $A=(a_{j,k-l+1})$, reducing to an equivalent system having the same number 
of solutions, in which the new coefficient matrix $A'$ has upper triangular form and $r'\le r$ 
non-vanishing rows.\par

A simple paucity result is provided by our first theorem. It is useful here and elsewhere to 
introduce the auxiliary quantity
\begin{equation}\label{1.5}
w(\bfvarphi)=\tfrac{1}{2}k(k+1)-k_1-k_2-\ldots -k_r.
\end{equation}
Here and throughout this paper, implicit constants in the notations of Landau and 
Vinogradov may depend on $\eps$, $k$, and the coefficients of $\bfvarphi$. 

\begin{theorem}\label{theorem1.1} Let $\varphi_1,\ldots ,\varphi_r$ be symmetric 
polynomials of the shape \eqref{1.3} having respective degrees $k_1,\ldots ,k_r$ satisfying 
\eqref{1.4}. Then for each $\eps>0$,
\[
N_k(X;\bfvarphi)=T_k(X)+O(X^{w(\bfvarphi)+1+\eps}).
\]
In particular, when $k_1+\ldots +k_r\ge \tfrac{1}{2}k(k-1)+2$, one has
\[
N_k(X;\bfvarphi)=k!X^k+O(X^{k-1+\eps}).
\]
\end{theorem}

A specialisation of the system \eqref{1.1} illustrates the kind of results made available by 
Theorem \ref{theorem1.1}. Fix a choice of coefficients $a_{jl}\in \dbZ$ for $1\le l\le k-r$ 
and $k-r+1\le j\le k$, and denote by $M_{k,r}(X;\bfa)$ the number of integral solutions of 
the simultaneous equations
\begin{equation}\label{1.6}
\sig_j(\bfx)+\sum_{l=1}^{k-r}a_{jl}\sig_l(\bfx)
=\sig_j(\bfy)+\sum_{l=1}^{k-r}a_{jl}\sig_l(\bfy)\quad (k-r+1\le j\le k),
\end{equation}
in variables $\bfx=(x_1,\ldots ,x_k)$ and $\bfy=(y_1,\ldots ,y_k)$ with $1\le x_i,y_i\le X$.

\begin{corollary}\label{corollary1.2} Suppose that $k,r\in \dbN$ and $(k-r)(k-r+1)<2k-2$. 
Then there is a paucity of non-diagonal solutions in the system of equations \eqref{1.6}. In 
particular, for each $\eps>0$ one has
\[
M_{k,r}(X;\bfa)=T_k(X)+O(X^{\frac{1}{2}(k-r)(k-r+1)+1+\eps}).
\]
\end{corollary}

Results analogous to those of Theorem \ref{theorem1.1} and Corollary \ref{corollary1.2} 
for diagonal Diophantine systems involving $r$ equations are almost always limited to 
systems possessing only $r+1$ pairs of variables, and these we now describe. In the case 
$r=1$, Hooley applied sieve methods to investigate the equation
\[
x_1^k+x_2^k=y_1^k+y_2^k,
\]
with $k\ge 3$, and ultimately established the paucity of non-diagonal solutions with a power 
saving (see \cite{Hoo1981, Hoo1996}). Strong conclusions have been derived when $k=3$ 
by Heath-Brown \cite{HB1997} using ideas based on quadratic forms, enhancing earlier work 
of the author that extends from cubes to general cubic polynomials \cite{Woo2000}. When 
$r=2$, $k_1<k_2$ and $k_2\ge 3$, the paucity of non-diagonal solutions has been 
established for the pair of simultaneous equations
\[
\left. \begin{aligned}
x_1^{k_1}+x_2^{k_1}+x_3^{k_1}&=y_1^{k_1}+y_2^{k_1}+y_3^{k_1}\\
x_1^{k_2}+x_2^{k_2}+x_3^{k_2}&=y_1^{k_2}+y_2^{k_2}+y_3^{k_2}
\end{aligned}\right\}.
\]
Sharp results are available in the case $(k_1,k_2)=(1,3)$ (see \cite{VW1995}), with 
non-trivial conclusions available for the exponent pair $(1,k)$ when $k\ge 3$ (see 
\cite{Gre1997, SW1997}). The cases $(2,3)$ and $(2,4)$ were tackled successfully via affine 
slicing methods (see \cite{TW1999, Woo1996}), with the remaining cases of this type 
covered by Salberger \cite{Sal2007} using variants of the determinant method. Most other 
examples in which it is known that there is a paucity of non-diagonal solutions are closely 
related to the Vinogradov system of equations
\[
x_1^j+\ldots +x_{r+1}^j=y_1^j+\ldots +y_{r+1}^j\quad (1\le j\le r).
\]
The paucity problem has been solved here by Vaughan and the author \cite{VW1997}, with 
similar conclusions when the equation of degree $r$ is replaced by one of degree $r+1$. 
Recent work \cite{Woo2021} shows that the missing equation of degree $r$ in this last result 
can be replaced by one of degree $r-d$, provided that $d$ is not too large. Meanwhile, when 
the exponents $k_j$ satisfy \eqref{1.4}, results falling just short of paucity have been 
obtained in general for systems of the shape
\[
x_1^{k_j}+\ldots +x_{r+1}^{k_j}=y_1^{k_j}+\ldots +y_{r+1}^{k_j}\quad (1\le j\le r),
\]
(see \cite{PW2002, Woo1993}). In the special case $\mathbf k=(1,3,\ldots ,2r-1)$, Br\"udern 
and Robert \cite{BR2015} have even established the desired paucity result. We should note 
also that the existence of non-diagonal solutions, often exhibited by remarkable parametric 
formulae, has long been the subject of investigation, as recorded in Gloden's book 
\cite{Glo1944}, with notable recent contributions by Choudhry \cite{Cho1991}.\par  

Taking $k$ large and $r=k-\lfloor \sqrt{2k}\rfloor +1$, we see that \eqref{1.6} constitutes 
a Diophantine system of $r$ equations in $k=r+\sqrt{2r}+O(1)$ pairs of variables having a 
paucity of non-diagonal solutions. Corollary \ref{corollary1.2} therefore exhibits a large class 
of Diophantine systems in which the barrier described in the previous paragraph is 
emphatically surmounted. There are two exceptions to the rule noted in that paragraph. First, 
Salberger and the author \cite[Corollary 1.4 and Theorem 5.2]{SW2010} have established 
the paucity of non-diagonal solutions in situations where the underlying equations have very 
large degree in terms of the number of variables. Second, work of Bourgain 
{\it et al.}~\cite[Theorem 26]{BGKS2014} and Heap 
{\it et al.}~\cite[Theorem 1.2]{HSW2021} examines systems of equations determined by 
relations of divisor type having the shape
\begin{equation}\label{1.7}
(x_1+\tet)(x_2+\tet)\cdots (x_k+\tet)=(y_1+\tet)(y_2+\tet)\cdots (y_k+\tet),
\end{equation}
in which $\tet\in \dbC$ is algebraic of degree $d$ over $\dbQ$. When $d\le k$, this relation 
generates $d$ independent symmetric Diophantine equations, and it is shown that the 
number of integral solutions of \eqref{1.7} with $1\le x_i,y_i\le X$ is asymptotically 
$T_k(X)+O(X^{k-d+1+\eps})$. In particular, when $d=2$, we obtain a pair of 
simultaneous Diophantine equations in $k$ variables having a paucity of non-diagonal 
solutions. For example, when $k=4$ and $\tet=\sqrt{-1}$, we obtain the system
\[
\left.{\begin{aligned}
x_1x_2x_3x_4-&x_1x_2-x_2x_3-x_3x_4-x_4x_1-x_2x_4-x_1x_3\\
&=y_1y_2y_3y_4-y_1y_2-y_2y_3-y_3y_4-y_4y_1-y_2y_4-y_1y_3\\
x_1x_2x_3+x_2&x_3x_4+x_3x_4x_1+x_4x_1x_2-x_1-x_2-x_3-x_4\\
&=y_1y_2y_3+y_2y_3y_4+y_3y_4y_1+y_4y_1y_2-y_1-y_2-y_3-y_4
\end{aligned}}\right\} ,
\]
having $4!X^4+O(X^{3+\eps})$ integral solutions with $1\le x_i,y_i\le X$.\par

We have avoided discussion of systems \eqref{1.1} containing the equation
\begin{equation}\label{1.8}
x_1x_2\cdots x_k=y_1y_2\cdots y_k.
\end{equation}
Here, when $r\ge 2$, one may parametrise the solutions of \eqref{1.8} in the shape
\begin{align*}
x_1=\alp_1\alp_2\cdots \alp_k,\quad x_2&=\bet_1\bet_2\cdots \bet_k,\quad \ldots ,\quad 
x_k=\ome_1\ome_2\cdots \ome_k,\\
y_1=\alp_1\bet_1\cdots \ome_1,\quad y_2&=\alp_2\bet_2\cdots \ome_2,\quad \ldots ,
\quad y_k=\alp_k\bet_k\cdots \ome_k,
\end{align*}
to provide a paucity result by simple elimination. The reader will find all of the ideas 
necessary to complete this elementary exercise in \cite{VW1995}.\par

This memoir is organised as follows. In \S2 we derive a multiplicative relation amongst the 
variables $\bfx,\bfy$ of the system \eqref{1.1}. This may be applied to obtain Theorem 
\ref{theorem1.1} and Corollary \ref{corollary1.2}. The polynomials $\varphi_j(\bfz)$ that 
are the subject of Theorem \ref{theorem1.1} are integral linear combinations of the 
elementary symmetric polynomials $\sig_1(\bfz),\ldots ,\sig_k(\bfz)$. In \S3 we examine 
the extent to which our methods are applicable when the polynomials $\varphi_j(\bfz)$ are 
permitted to depend non-linearly on $\sig_1(\bfz),\ldots ,\sig_k(\bfz)$.\par

Our basic parameter is $X$, a sufficiently large positive number. Whenever $\eps$ appears 
in a statement, either implicitly or explicitly, we assert that the statement holds for each 
$\eps>0$. We make frequent use of vector notation in the form $\bfx=(x_1,\ldots,x_k)$. 
Here, the dimension $k$ will be evident to the reader from the ambient context.\medskip
 
\noindent {\bf Acknowledgements:} The author's work is supported by NSF grants 
DMS-2001549 and DMS-1854398. We thank the referee for useful comments.

\section{Multiplicative relations from symmetric polynomials} Our initial objective in this 
section is to obtain a multiplicative relation between the variables underlying the system 
\eqref{1.1}. Let $\varphi_1,\ldots ,\varphi_r$ be polynomials of the shape \eqref{1.3} 
having respective degrees $k_1,\ldots ,k_r$ satisfying \eqref{1.4}. We define the 
complementary set of exponents $\calR=\calR(\bfvarphi)$ by putting
\begin{equation}\label{2.1}
\calR(\bfvarphi)=\{1,\ldots ,k\}\setminus \{k_1,\ldots ,k_r\}.
\end{equation}
The counting function $N_k(X;\bfvarphi)$ remains unchanged when we replace the 
polynomials $\varphi_j$ $(1\le j\le r)$ by any collection of linearly independent integral 
linear combinations, and thus we may suppose that these polynomials take the shape
\begin{equation}\label{2.2}
\varphi_j(\bfz)=a_j\sig_{k_j}(\bfz)-\sum_{\substack{1\le l<k_j\\ l\in \calR}}
b_{jl}\sig_l(\bfz)\quad (1\le j\le r),
\end{equation}
where $a_j\in \dbZ\setminus \{0\}$ and $b_{jl}\in \dbZ$.\par

Given a solution $\bfx,\bfy$ of the system \eqref{1.1} counted by $N_k(X;\bfvarphi)$, we 
define the integers $h_l=h_l(\bfx,\bfy)$ for $l\in \calR$ by putting
\begin{equation}\label{2.3}
h_l(\bfx,\bfy)=\sig_l(\bfx)-\sig_l(\bfy).
\end{equation}
Thus, since $1\le x_i,y_i\le X$ $(1\le i\le k)$, one has $|h_l(\bfx,\bfy)|\le 2^kX^l$ 
$(l\in \calR)$. By making use of the relations \eqref{2.2} and \eqref{2.3}, the system 
\eqref{1.1} becomes
\begin{equation}\label{2.4}
a_j\left(\sig_{k_j}(\bfx)-\sig_{k_j}(\bfy)\right) =\sum_{\substack{1\le l<k_j\\ l\in \calR}}
b_{jl}h_l\quad (1\le j\le r).
\end{equation}
Then, by wielding \eqref{1.2} in combination with \eqref{2.3} and \eqref{2.4}, we obtain
\begin{align}
\prod_{i=1}^k(t+x_i)-\prod_{i=1}^k(t+y_i)&=\sum_{m=0}^kt^{k-m}
\left(\sig_m(\bfx)-\sig_m(\bfy)\right) \notag \\
&=\sum_{m\in \calR}h_mt^{k-m}+\sum_{j=1}^ra_j^{-1}t^{k-k_j}
\sum_{\substack{1\le l<k_j\\ l\in \calR}}b_{jl}h_l.\label{2.5}
\end{align}
Put $A=a_1a_2\cdots a_r$ and $c_j=A/a_j$ $(1\le j\le r)$. Also, define
\[
\psi_l(t)=At^{k-l}+\sum_{\substack{1\le j\le r\\ k_j>l}}c_jb_{jl}t^{k-k_j}\quad (l\in \calR),
\]
and then set
\begin{equation}\label{2.6}
\Psi(t;\bfh)=\sum_{l\in \calR}h_l\psi_l(t).
\end{equation}
Notice here that the polynomial $\Psi(t;\bfh)$ has integral coefficients and degree at most 
$k-1$ with respect to $t$. Moreover, it follows from \eqref{2.5} that
\begin{equation}\label{2.7}
A\biggl( \prod_{i=1}^k(t+x_i)-\prod_{i=1}^k(t+y_i)\biggr) =\Psi(t;\bfh).
\end{equation}

\par We may now record the multiplicative relation employed in our proof of Theorem 
\ref{theorem1.1}. Suppose that $\bfx,\bfy$ is an integral solution of the system 
\eqref{1.1}, where $\varphi_1,\ldots ,\varphi_r$ are symmetric polynomials as described 
above. Then, with $h_l=\sig_l(\bfx)-\sig_l(\bfy)$ $(l\in \calR)$, we deduce by substituting 
$t=-y_j$ into \eqref{2.7} that
\begin{equation}\label{2.8}
A\prod_{i=1}^k(x_i-y_j)=\Psi(-y_j;\bfh)\quad (1\le j\le k).
\end{equation}

\begin{proof}[The proof of Theorem \ref{theorem1.1}] We divide the solutions of the 
system \eqref{1.1} with $1\le x_i,y_i\le X$ into two types. A solution $\bfx,\bfy$ will be 
called {\it potentially diagonal} when $\{x_1,\ldots ,x_k\}=\{y_1,\ldots ,y_k\}$, and 
{\it non-diagonal} when, for some index $j$ with $1\le j\le k$, one has either 
$y_j\not\in \{x_1,\ldots ,x_k\}$ or $x_j\not \in \{y_1,\ldots ,y_k\}$.\par

We first consider potentially diagonal solutions $\bfx,\bfy$. Suppose, if possible, that 
$\bfx,\bfy$ satisfies the condition that the polynomial $\Psi(t;\bfh)$ defined in \eqref{2.6} is 
identically zero as a polynomial in $t$. Since $A\ne 0$, it follows from \eqref{2.7} that
\[
\prod_{i=1}^k(t+x_i)=\prod_{i=1}^k(t+y_i).
\]
The roots of the polynomials on the left and right hand sides here must be identical, and so 
too must be their respective multiplicities. Thus $(x_1,\ldots ,x_k)$ is a permutation of 
$(y_1,\ldots ,y_k)$, and it follows that the number of solutions $\bfx,\bfy$ of this type 
counted by $N_k(X;\bfvarphi)$ is precisely $T_k(X)$.\par

Suppose next that $\bfx,\bfy$ is a potentially diagonal solution with $\Psi(t;\bfh)$ not 
identically zero, where $\bfh$ is defined via \eqref{2.3}. In this situation, there are distinct 
integers $w_1,\ldots ,w_s$, for some integer $s$ with $1\le s\le k$, such that 
\begin{equation}\label{2.9}
\{x_1,\ldots ,x_k\}=\{y_1,\ldots ,y_k\}=\{w_1,\ldots ,w_s\}.
\end{equation}
Consider any tuple of integers $\bfh$ with $|h_l|\le 2^kX^l$ $(l\in \calR)$ for which the  
polynomial $\Psi(t;\bfh)$ is not identically zero. This polynomial has degree at most $k-1$, 
and hence there is an integer $\xi$ with $1\le \xi\le k$ for which the integer 
$\Tet=\Psi(\xi;\bfh)$ is non-zero. We see from \eqref{2.7} that 
$(\xi+w_1)\cdots (\xi+w_s)$ divides $\Tet$. Thus, an elementary divisor function estimate 
(see for example \cite[Theorem 317]{HW2008}) shows the number of possible choices for 
$\xi+w_1,\ldots ,\xi+w_s$ to be at most
\[
\sum_{\substack{d_1,\ldots ,d_s\in \mathbb N\\ d_1\cdots d_s|\Tet}}1\le 
\Bigl( \sum_{d|\Tet}1\Bigr)^s=d(|\Tet|)^s\ll |\Tet|^\eps.
\]
Since $\Tet=O(X^k)$, we see that there are at most $O(X^\eps)$ possible choices for 
$\xi+w_1,\ldots ,\xi+w_s$, and hence also for $w_1,\ldots ,w_s$. From here, the 
relation \eqref{2.9} implies that for these fixed choices of $\bfh$, the number of possible 
choices for $\bfx$ and $\bfy$ is also $O(X^\eps)$. On recalling \eqref{1.5}, we discern 
that the total number of choices for the tuple $\bfh$ with $|h_l|\le 2^kX^l$ $(l\in \calR)$ is 
$O(X^{w(\bfvarphi)})$. Let $T_k^*(X;\bfvarphi)$ denote the number of potentially 
diagonal solutions $\bfx,\bfy$ not counted by $T_k(X)$. For each of the 
$O(X^{w(\bfvarphi)})$ choices for $\bfh$ associated with these solutions, we have shown 
that there are $O(X^\eps)$ solutions $\bfx,\bfy$, whence
\begin{equation}\label{2.10}
T_k^*(X;\bfvarphi)\ll X^{w(\bfvarphi)+\eps}.
\end{equation}

\par Finally, suppose that $\bfx,\bfy$ is a non-diagonal solution of the system \eqref{1.1}. 
By invoking symmetry (twice), we may suppose that $y_k\not\in \{x_1,\ldots ,x_k\}$, 
whence \eqref{2.8} shows the integer $\Tet =\Psi(-y_k;\bfh)$ to be non-zero. There are 
$O(X^{w(\bfvarphi)})$ possible choices for $\bfh$ and $O(X)$ possible choices for $y_k$ 
corresponding to this situation. Fixing any one such, and noting that $\Tet=O(X^k)$, an 
elementary divisor function estimate again shows that there are $O(X^\eps)$ possible 
choices for $x_1-y_k,\ldots ,x_k-y_k$ satisfying \eqref{2.8}. Fix any one such choice for 
these divisors. Since $y_k$ is already fixed, it follows that $x_1,\ldots ,x_k$ are likewise 
fixed.\par

It remains to determine $y_1,\ldots ,y_{k-1}$. Since the tuple $\bfh$ has already been 
fixed and $\Psi(t;\bfh)$ has degree at most $k-1$ with respect to $t$, the polynomial
\[
A\prod_{i=1}^k(x_i-t)-\Psi(-t;\bfh)
\]
has degree $k$ with respect to $t$, and has all of its coefficients already fixed. It therefore 
follows from \eqref{2.8} that there are at most $k$ choices for each of the variables 
$y_1,\ldots ,y_{k-1}$. Let $T_k^\dagger(X;\bfvarphi)$ denote the number of non-diagonal 
solutions $\bfx,\bfy$ counted by $N_k(X;\bfvarphi)$. Then we may conclude that
\begin{equation}\label{2.11}
T_k^\dagger(X;\bfvarphi)\ll X^{w(\bfvarphi)+1+\eps}.
\end{equation}

\par On recalling our opening discussion together with the estimates \eqref{2.10} and 
\eqref{2.11}, we arrive at the upper bound
\begin{equation}\label{2.12}
N_k(X;\bfvarphi)-T_k(X)=T_k^*(X;\bfvarphi)+T_k^\dagger(X;\bfvarphi)\ll 
X^{w(\bfvarphi)+1+\eps}.
\end{equation}
This delivers the first conclusion of Theorem \ref{theorem1.1}. On recalling the definition 
\eqref{1.5} of $w(\bfvarphi)$, it follows that when 
$k_1+\ldots +k_r\ge \tfrac{1}{2}k(k-1)+2$, one has 
$w(\bfvarphi)\le k-2$, and thus the second conclusion of Theorem \ref{theorem1.1} is 
immediate from \eqref{2.12} and the asymptotic formula $T_k(X)=k!X^k+O(X^{k-1})$.
\end{proof}

\begin{proof}[The proof of Corollary \ref{corollary1.2}] Write
\[
\varphi_j(\bfz)=\sig_j(z_1,\ldots ,z_k)+\sum_{l=1}^{k-r}a_{jl}\sig_l(z_1,\ldots ,z_k)\quad 
(k-r+1\le j\le k).
\]
Then we see that the system \eqref{1.6} is comprised of symmetric polynomials having 
degrees $k-r+1,\ldots ,k$. For this system, it follows from \eqref{1.5} that
\[
w(\bfvarphi)=\tfrac{1}{2}k(k+1)-\sum_{j=k-r+1}^k j=\tfrac{1}{2}(k-r)(k-r+1).
\]
We therefore deduce from Theorem \ref{theorem1.1} that
\[
M_{k,r}(X;\bfa)=N_k(X;\bfvarphi)=T_k(X)+O(X^{\frac{1}{2}(k-r)(k-r+1)+1+\eps}),
\]
and so there is a paucity of non-diagonal solutions in the system \eqref{1.6} provided that 
$(k-r)(k-r+1)<2k-2$. This completes the proof of the corollary.
\end{proof}

\section{Non-linear variants} The system of Diophantine equations \eqref{1.1} underlying 
Theorem \ref{theorem1.1} and Corollary \ref{corollary1.2} possesses features inherently 
linear in nature. Indeed, as is evident from the discussion initiating \S2, we are able to 
restrict attention to integral linear combinations of elementary symmetric polynomials of the 
shape \eqref{2.2}. This is a convenient but not essential simplification, as we now explain.

\par We now describe the symmetric polynomials presently in our field of view. Recall the 
exponents $k_1,\ldots ,k_r$ satisfying \eqref{1.4}, the complementary set of exponents 
$\calR=\calR(\bfvarphi)$ defined in \eqref{2.1}, and the definition \eqref{1.5} of 
$w(\bfvarphi)$. Putting $R=\text{card}(\calR)$, we label the elements of $\calR$ so that 
$\calR=\{l_1,\ldots ,l_R\}$. We consider symmetric polynomials in the variables 
$\bfz=(z_1,\ldots ,z_k)$ of the shape
\begin{equation}\label{3.1}
\varphi_j(\bfz)=a_j\sig_{k_j}(\bfz)-\Ups_j(\sig_{l_1}(\bfz),\ldots ,\sig_{l_R}(\bfz))\quad 
(1\le j\le r),
\end{equation}
where $a_j\in \dbZ\setminus\{0\}$ and $\Ups_j\in \dbZ[s_1,\ldots ,s_R]$.

\begin{theorem}\label{theorem3.1} Let $\varphi_1,\ldots ,\varphi_r$ be symmetric 
polynomials of the shape \eqref{3.1} with respective degrees $k_1,\ldots ,k_r$ satisfying 
\eqref{1.4}. Then
\[
N_k(X;\bfvarphi)=T_k(X)+O(X^{2w(\bfvarphi)+1+\eps}).
\]
In particular, when $k_1+\ldots +k_r\ge \tfrac{1}{2}k^2+1$, one has
\[
N_k(X;\bfvarphi)=k!X^k+O(X^{k-1+\eps}).
\]
\end{theorem}

A specialisation again makes for accessible conclusions. Fix the polynomials 
$\Ups_j\in \dbZ[s_1,\ldots ,s_R]$ for $k-r+1\le j\le k$, and denote by $L_{k,r}(X;\bfUps)$ 
the number of solutions of the simultaneous equations
\begin{align}
\sig_j(\bfx)+\Ups_j(\sig_1(\bfx),\ldots ,\sig_{k-r}(\bfx))=\sig_j(\bfy)+\Ups_j(\sig_1&(\bfy),
\ldots ,\sig_{k-r}(\bfy)) \notag \\
&(k-r+1\le j\le k),\label{3.2}
\end{align}
in variables $\bfx=(x_1,\ldots ,x_k)$ and $\bfy=(y_1,\ldots ,y_k)$ with $1\le x_i,y_i\le X$.

\begin{corollary}\label{corollary3.2} Suppose that $k,r\in \dbN$ satisfy $(k-r)(k-r+1)<k-1$. 
Then there is a paucity of non-diagonal solutions in the system \eqref{3.2}. In particular,
\[
L_{k,r}(X;\bfUps)=T_k(X)+O(X^{(k-r)(k-r+1)+1+\eps}).
\]
\end{corollary}

\begin{proof}[The proof of Theorem \ref{theorem3.1}] Our present strategy is very similar 
to that wrought against Theorem \ref{theorem1.1}. Given a solution $\bfx,\bfy$ of the 
system \eqref{1.1} counted by $N_k(X;\bfvarphi)$, define $h_m(\bfz)=\sig_{l_m}(\bfz)$ 
for $\bfz\in \{\bfx,\bfy\}$ and $1\le m\le R$. In view of \eqref{3.1}, the system \eqref{1.1} now 
becomes
\begin{equation}\label{3.3}
a_j(\sig_{k_j}(\bfx)-\sig_{k_j}(\bfy))=\Ups_j(\bfh(\bfx))-\Ups_j(\bfh(\bfy))\quad 
(1\le j\le r).
\end{equation}
In the present non-linear scenario there may be some index $j$ for which the integer on 
the right hand side of \eqref{3.3} is not equal to $\Ups_j(\bfh(\bfx)-\bfh(\bfy))$. The 
argument here therefore contains extra complications, with weaker quantitative conclusions 
than Theorem \ref{theorem1.1}.\par

Put $A=a_1a_2\cdots a_r$ and $c_j=A/a_j$ $(1\le j\le r)$. Then, by applying the identity 
\eqref{1.2} together with \eqref{3.3}, we obtain for $1\le j\le r$ the relation
\begin{equation}\label{3.4}
A\Bigl( \prod_{i=1}^k(t+x_i)-\prod_{i=1}^k(t+y_i)\Bigr)=\Psi(t;\bfh)-\Psi(t;\bfg),
\end{equation}
where $\bfh=(h_1(\bfx),\ldots ,h_R(\bfx))$, $\bfg=(h_1(\bfy),\ldots ,h_R(\bfy))$, and we 
write
\[
\Psi(t;\bfe)=A\sum_{m\in \calR}e_mt^{k-m}+\sum_{j=1}^rc_jt^{k-k_j}\Ups_j(\bfe).
\]
If $\bfx,\bfy$ is a solution of \eqref{1.1} counted by $N_k(X;\bfvarphi)$, then 
with $h_m=\sig_{l_m}(\bfx)$ and $g_m=\sig_{l_m}(\bfy)$ for $1\le m\le R$, we deduce 
by setting $t=-y_j$ in \eqref{3.4} that
\begin{equation}\label{3.5}
A\prod_{i=1}^k(x_i-y_j)=\Psi(-y_j;\bfh)-\Psi(-y_j;\bfg)\quad (1\le j\le k).
\end{equation}

\par We now follow the path already trodden in the proof of Theorem \ref{theorem1.1} 
presented in \S2. Suppose first that $\bfx,\bfy$ is a potentially diagonal solution. When the 
polynomial $\Psi(-t;\bfh)-\Psi(-t;\bfg)$ is identically zero, we find that $\bfx,\bfy$ is counted 
by $T_k(X)$. Meanwhile, when instead this polynomial is not identically zero, a divisor 
function argument shows that for each fixed choice of $\bfh$ and $\bfg$, there are 
$O(X^\eps)$ possible choices for $\bfx$ and $\bfy$. We have 
$|h_m(\bfx)|\le 2^kX^{l_m}$ and $|h_m(\bfy)|\le 2^kX^{l_m}$, so the total number of 
choices for $(h_1,\ldots ,h_R)$ and $(g_1,\ldots ,g_R)$ is $O(X^{2w(\bfvarphi)})$, where 
$w(\bfvarphi)$ is defined by \eqref{1.5}. Thus, the total number of potentially diagonal 
solutions is equal to $T_k(X)+O(X^{2w(\bfvarphi)+\eps})$.\par

Suppose next that $\bfx,\bfy$ is a non-diagonal solution of the system \eqref{1.1}. By 
symmetry, we may again suppose that $y_k\not\in \{x_1,\ldots ,x_k\}$, and \eqref{3.5} 
shows that the integer $\Tet =\Psi(-y_j;\bfh)-\Psi(-y_j;\bfg)$ is non-zero. There are 
$O(X^{2w(\bfvarphi)})$ possible choices for $\bfh$ and $\bfg$, and $O(X)$ possible 
choices for $y_k$ in this scenario. Fix any one such choice, and note from \eqref{3.5} that 
$\Tet=O(X^k)$. An elementary divisor function estimate shows there to be $O(X^\eps)$ 
choices for $x_1-y_k,\ldots ,x_k-y_k$ satisfying \eqref{3.5}. Fixing any one such choice 
fixes the integers $x_1,\ldots ,x_k$. An argument essentially identical to that applied in the 
proof of Theorem \ref{theorem1.1} shows from here that there are $O(1)$ possible 
choices for $y_1,\ldots ,y_{k-1}$. Hence, the number of non-diagonal solutions of the 
system \eqref{1.1} is $O(X^{2w(\bfvarphi)+1+\eps})$.\par

Combining the two contributions to $N_k(X;\bfvarphi)$ that we have obtained yields
\[
N_k(X;\bfvarphi)-T_k(X)\ll X^{2w(\bfvarphi)+1+\eps},
\]
confirming the first conclusion of Theorem \ref{theorem3.1}. The definition \eqref{1.5} of 
$w(\bfvarphi)$, moreover, implies that when $k_1+\ldots +k_r\ge \tfrac{1}{2}k^2+1$, one 
has $2w(\bfvarphi)\le k-2$. The second conclusion of Theorem \ref{theorem3.1} therefore 
follows directly from the first.
\end{proof}

\begin{proof}[The proof of Corollary \ref{corollary3.2}] When $k-r+1\le j\le k$, write
\[
\varphi_j(\bfz)=\sig_j(z_1,\ldots ,z_k)+\Ups_j(\sig_1(z_1,\ldots ,z_k),\ldots ,
\sig_{k-r}(z_1,\ldots ,z_k)).
\]
Then we see from \eqref{1.5} that the system \eqref{3.2} is comprised of symmetric 
polynomials with $w(\bfvarphi)=\tfrac{1}{2}(k-r)(k-r+1)$. Theorem \ref{theorem3.1} 
therefore shows that
\[
L_{k,r}(X;\bfUps)=N_k(X;\bfvarphi)=T_k(X)+O(X^{(k-r)(k-r+1)+1+\eps}),
\]
and provided that $(k-r)(k-r+1)<k-1$, there is a paucity of non-diagonal solutions in the 
system \eqref{3.2}. This completes the proof of the corollary.
\end{proof}

\bibliographystyle{amsbracket}
\providecommand{\bysame}{\leavevmode\hbox to3em{\hrulefill}\thinspace}

\end{document}